\documentclass[10pt]{article}
\usepackage{graphicx}
\usepackage{mathpazo}
\usepackage{latexsym}
\usepackage{a4wide}
\usepackage{mathrsfs}
\usepackage{amsmath}
\usepackage{amssymb}
\usepackage{tikz}
\usepackage{enumitem}
\usetikzlibrary{calc}

\newtheorem{thm}{Theorem}[section]

\newtheorem{qst}[thm]{Question}
\newtheorem{cor}[thm]{Corollary}

\newtheorem{rem}[thm]{Remark}
\newtheorem{lemma}[thm]{Lemma}

\numberwithin{equation}{section}
\newenvironment{proof}{\vspace{-0.1cm}\noindent\textbf{Proof:}}{$\Box$\\}

\tikzstyle{state}=[circle,thick,draw=black!80]
\tikzstyle{term}=[rectangle,thick,draw=black!80]
\tikzstyle{andsoon}=[rectangle,thick,draw=white!80,fill=white!20]
\tikzstyle{andsoonfill}=[rectangle,rounded corners, text width=8em, text centered,thick,draw=black!80,fill=black!10]
\tikzstyle{nameex}=[rectangle,thick]

\title{Games characterizing limsup functions and Baire class 1 functions}

\author{
M\'{a}rton Elekes\footnote{Alfr\'{e}d R\'{e}nyi Institute of Mathematics, Reáltanoda u. 13-15, 1053 Budapest, Hungary, AND Eötvös Loránd University, Institute of Mathematics, Pázmány P. sétány 1/c, 1117 Budapest, Hungary. URL: http://www.renyi.hu/$\sim$emarci E-mail: elekes.marton@renyi.hu.} \and J\'{a}nos Flesch\footnote{Maastricht University, SBE, P.O.Box 616, 6200 MD, The Netherlands. E-mail: j.flesch@maastrichtuniversity.nl.} \and Viktor Kiss\footnote{Alfr\'ed R\'enyi Institute of Mathematics, Re\'altanoda u. 13--15, H-1053 Budapest, Hungary. E-mail: kiss.viktor@renyi.hu.} \and Don\'{a}t Nagy\footnote{E\"otv\"os Lor\'and University, Institute of Mathematics, P\'azm\'any P\'eter s. 1/c, 1117 Budapest, Hungary. E-mail: m1nagdon@gmail.com.} \and M\'{a}rk Po\'{o}r\footnote{E\"otv\"os Lor\'and University, Institute of Mathematics, P\'azm\'any P\'eter s. 1/c, 1117 Budapest, Hungary. E-mail: sokmark@gmail.com} \and Arkadi Predtetchinski\footnote{Maastricht University, SBE, P.O.Box 616, 6200 MD, The Netherlands. E-mail: a.predtetchinski@maastrichtuniversity.nl.}}

\begin{document}

\maketitle

\renewcommand{\thefootnote}{}
\footnotetext{The first, third, fourth and fifth authors were supported by the National Research, Development and Innovation Office -- NKFIH, grants no.~113047, 129211 and 124749. The third author was supported by the National Research, Development and Innovation Office -- NKFIH, grant no.~128273. The fifth author was supported by the ÚNKP-19-3 New National Excellence Program of the Ministry for Innovation and Technology.}



\begin{abstract}
\noindent We consider a real-valued function $f$ defined on the set of infinite branches $X$ of a countably branching pruned tree $T$. The function $f$ is said to be a \textit{limsup function} if there is a function $u \colon T \to \mathbb{R}$ such that $f(x) = \limsup_{t \to \infty} u(x_{0},\dots,x_{t})$ for each $x \in X$. We study a game characterization of limsup functions, as well as a novel game characterization of functions of Baire class 1.
\end{abstract}

\tableofcontents\medskip


\section{Introduction}
Throughout the paper, let $T$ be a pruned tree on a non-empty countable set $A$, and $X$ be the set of its infinite branches. We say that $f : X \to \mathbb{R}$ is a \textit{limsup function} if there exists a function $u : T \to \mathbb{R}$ such that, for every $x \in X$,
\begin{equation}\label{eqn.limsup}
f(x) = \limsup_{t \rightarrow \infty} u(x_{0},\dots,x_{t}).
\end{equation}

Payoff evaluations of limsup type are ubiquitous in gambling theory (\cite{Dubins[2014]}), in the theory of dynamic games (\cite{Levy[2017]}), and in computer science (\cite{Bru[2017]}). Limsup payoff evaluation expresses the decision maker's preference to receive high payoff infinitely often.   

We first relate limsup functions to certain well-known classes of functions. In fact, $f$ is a limsup function precisely if it is a pointwise limit of a descending sequence of lower semicontinuous functions. Pointwise limits of a descending sequence of lower semicontinuous functions have been studied e.g. in Hausdorff \cite{Hausdorff[2005]}. It follows in particular that $f$ is a limsup function exactly if its subgraph is a $\mathbf{\Pi}_{2}^{0}$  set (i.e. a $G_{\delta}$ set), and that the sum, the minimum, and the maximum of two limsup functions is a limsup function. We also deduce a characterization of Baire class 1 functions $f: X \to \mathbb{R}$: these are exactly the functions such that both $f$ and $-f$ are limsup functions.

The core of the paper is devoted to the study of two related games. The first one is the following:
\begin{center}
\begin{tabular}{cccccc}
{\rm I}& $x_{0}$ &         & $x_{1}$ &         & $\cdots$\\
{\rm II} &         & $v_{0}$ &         & $v_{1}$ & $\cdots$
\end{tabular}
\end{center}
The moves $x_{0},x_{1},\dots$ of Player I are points of $A$ such that $(x_{0},\dots,x_{t}) \in T$ for each $t \in \mathbb{N}$. The moves $v_{0},v_{1},\dots$ of Player II are real numbers. The game starts with a move of Player I, $x_{0}$. Having observed $x_{0}$, Player II chooses $v_{0}$. Having observed $v_{0}$, Player I chooses $x_{1}$, and so on. In this fashion the players produce a run of the game, $(x_{0},v_{0},x_{1},v_{1},\dots)$. Player II wins the run if $f(x_{0},x_{1},\dots) = \limsup v_{t}$. We denote this game by $\Gamma(f)$.

As we will see in Lemma \ref{l:player II wins iff f is limsup}, Player II has a winning strategy in $\Gamma(f)$ precisely when $f$ is a limsup function. Whether Player I has a winning strategy in $\Gamma(f)$ turns out to be a more subtle question. We give a sufficient condition for Player I to have a winning strategy in $\Gamma(f)$, a condition that is also necessary if either $f$ is Borel measurable (more precisely, it suffices if the sets of the form $\{x \in X : f(x) \ge r \}$ are co-analytic), or if the range of $f$ contains no infinite strictly increasing sequence, in particular if $f$ takes only finitely many values. We also show that the game $\Gamma(f)$ is determined if $f$ is Borel measurable (again, it suffices if the sets of the form $\{x \in X : f(x) \ge r \}$ are co-analytic), but not in general.

The second game, denoted by $\Gamma'(f)$, is as follows:
\begin{center}
\begin{tabular}{cccccc}
{\rm I}& $x_{0}$ &                 & $x_{1}$ &                 & $\cdots$\\
{\rm II} &         & $(v_{0},w_{0})$ &         & $(v_{1},w_{1})$ & $\cdots$
\end{tabular}
\end{center}
This game is similar to $\Gamma(f)$ except that now the moves $(v_{0},w_{0}),(v_{1},w_{1}),\dots$ of Player II are pairs of real numbers. Player II wins in $\Gamma'(f)$ if $f(x_{0},x_{1},\dots) = \limsup v_{t} = \liminf w_{t}$. We denote this game by $\Gamma'(f)$.

Player II has a winning strategy in the game $\Gamma'(f)$ precisely when he has a winning strategy in both games $\Gamma(f)$ and $\Gamma(-f)$, which is the case exactly when $f$ is in Baire class 1. Moreover, the game $\Gamma'(f)$ is always determined. This result holds for any function $f$, whether or not $f$ is Borel measurable, and is established without the aid of Martin's determinacy.

The so-called eraser game characterizing Baire class 1 functions from the Baire space to itself was constructed in Duparc \cite{Duparc[2001]}. Carroy \cite{Carroy[2014]} showed that the eraser game is determined, and Kiss \cite{Kiss[2019]} generalized the characterization to functions of arbitrary Polish spaces. Game characterizations of several other classes of functions have been considered in Semmes \cite{Semmes[2008]}, Carroy \cite{Carroy[2014]}, and Nobrega \cite{Nobrega[2019]}.

Section \ref{secn.top} discusses characterizations of limsup functions. Section \ref{secn.gamelimsup} and Section \ref{secn.gamebaire} are devoted to the analysis of the games $\Gamma(f)$ and $\Gamma'(f)$, respectively.

\section{Characterizations of limsup functions}\label{secn.top}
 For $s \in T$, we let $O(s)$ denote the set of $x \in X$ such that $s$ is an initial segment of $x$. We refer to $O(s)$ as a cylinder set. We endow $X$ with its usual topology, generated by the base consisting of all cylinder sets. For a function $f : X \to \mathbb{R}$ write $\mathop{subgr}(f) = \{(x,r) \in X \times \mathbb{R}: f(x) \ge r\}$ to denote the \textit{subgraph} of $f$. For $r \in \mathbb{R}$, we write $\{f \ge r\} = \{x \in X: f(x) \ge r\}$, $\{f > r\} = \{x \in X: f(x) > r\}$ and $\{f=r\} = \{x\in X: f(x)= r\}$.


\begin{thm}\label{thm.charlimsup}
Consider a function $f : X \to \mathbb{R}$. The following conditions are equivalent:
\begin{enumerate}[label={\rm[C\arabic*]}]
\item \label{item.limsup} The function $f$ is a limsup function.
\item \label{item.converging} There is a sequence $g_{0}, g_{1},\dots$ of lower semicontinuous functions converging pointwise to $f$.
\item \label{item.decreasing} There is a non-increasing sequence $g_{0} \geq g_{1} \geq \cdots$ of lower semicontinuous functions converging pointwise to $f$.
\item \label{item.subgraph} The set $\mathop{subgr}(f)$ is a $\mathbf{\Pi}_{2}^{0}$ subset of $X \times \mathbb{R}$.
\item \label{item.contoursets} For each $r \in \mathbb{R}$, $\{f \ge r \}$ is a $\mathbf{\Pi}_{2}^{0}$ subset of $X$.
\end{enumerate}
\end{thm}

We remark that the functions satisfying condition \ref{item.contoursets} are sometimes  called semi-Borel class 2 (see Kiss \cite{Kiss[2017]}) or upper semi-Baire class 1 functions. The equivalence of the conditions \ref{item.converging}, \ref{item.decreasing}, \ref{item.subgraph}, \ref{item.contoursets} is in fact well known (see Hausdorff \cite{Hausdorff[2005]}). Below we prove the equivalence of conditions \ref{item.limsup}, \ref{item.converging}, and \ref{item.decreasing}.\bigskip

\noindent\textsc{Proof that \ref{item.limsup} implies \ref{item.converging}:} Let $f$ be a limsup function, and let $u$ be a function as in \eqref{eqn.limsup}. For $n \in \mathbb{N}$ let $g_{n}(x) = \sup\{u(x_{0},\dots,x_{t}):t \geq n\}$. $\Box$\bigskip

\noindent\textsc{Proof that \ref{item.converging} implies \ref{item.decreasing}:} Let $g_{n}$ be a sequence of lower semicontinuous functions converging pointwise to $f$. Define $g_{n}'(x) = \sup\{g_{m}(x):m \geq n\}$. This gives a non-increasing sequence of lower semicontinuous functions converging pointwise to $f$. $\Box$\bigskip

\noindent\textsc{Proof that \ref{item.decreasing} implies \ref{item.limsup}:} Consider a non-increasing sequence $g_{0} \geq g_{1} \geq \cdots$ of lower semicontinuous functions converging pointwise to $f$. We will also assume that for each $n \in \mathbb{N}$, the range of $g_{n}$ contains only reals of the form $z2^{-n}$ for $z \in \mathbb{Z}$. To see that this could be imposed without loss of generality, consider the functions $g_{n}'(x) = \min\{z2^{-n}: z \in \mathbb{Z}, g_{n}(x) \leq z2^{-n}\}$. Then $\{g_{n}'> z2^{-n}\}$ is the same as the set $\{g_{n} > z2^{-n}\}$, implying that $g_{n}'$ is lower semicontinuous. It is easy to see that $g_{0}'\geq g_{1}' \geq \cdots$ is a non-increasing sequence, and that it converges pointwise to $f$.

We define the function $u : T \rightarrow \mathbb{R}$. For $n \in \mathbb{N}$ and $r \in \mathbb{R}$ note that the set $\{g_n> r\}$ is an open set, because $g_{n}$ is assumed to be lower semicontinuous. Take a sequence $s \in T$. Define $R_{*}(s)$ to be the set of real numbers $r \in \mathbb{R}$ such that $O(s) \subseteq \bigcap_{n \in \mathbb{N}}\{g_n> r\}$. For $n \in \mathbb{N}$ define $R_{n}(s)$ to be the set of real numbers $r \in \mathbb{R}$ such that $O(s) \subseteq \{g_n>r\}$, and such that for no proper initial segment $s'$ of $s$ does it holds that $O(s')  \subseteq \{g_n>r\}$. (We remark that $R_{*}(s)$ is a half-line and the sets $R_{n}(s)$ are intervals.) Let $R(s)$ be the union of the sets $R_{*}(s),R_{0}(s),R_{1}(s),\dots$. Notice that the set $R(s)$ is bounded above by $\inf\{g_{0}(y):y \in O(s)\}$. If $R(s)$ is non-empty, we define $u(s) = \sup R(s)$. If $R(s)$ is empty, we let $u(s) = -{\rm length}(s)$.

We show that $u$ satisfies \eqref{eqn.limsup}. Thus fix an $x \in X$. We write $s_{t}$ to denote $(x_{0},\dots,x_{t})$ and let $\alpha = \limsup_{t \rightarrow \infty}u(s_{t})$. We must show that $f(x) = \alpha$.

We first show that $f(x) \leq \alpha$.

Take a real number $r$ with $r < f(x)$. We argue that $r \leq \alpha$.

For every $n \in \mathbb{N}$ it holds that $r < g_{n}(x)$, so $x \in \{g_n> r\}$. Let $t_{n}$ be the smallest $t \in \mathbb{N}$ for which $O(s_{t_{n}}) \subseteq \{g_n>r\}$. We distinguish two cases, depending on whether the sequence $t_{0}, t_{1},\dots$ is bounded or not. Suppose first the sequence $t_{0}, t_{1},\dots$ is unbounded. By the choice of $t_{n}$, we have $r \in R_{n}(s_{t_{n}})$, and hence $r \leq u(s_{t_{n}})$. We obtain $r \leq \alpha$, as desired. Suppose now that the sequence $t_{0}, t_{1},\dots$ is bounded, say $t_{n} \leq t$ for each $n \in \mathbb{N}$. Then $O(s_{t}) \subseteq \bigcap_{n \in \mathbb{N}}\{g_n>r\}$. Since for $k \geq t$ the cylinder $O(s_{k})$ is contained in $O(s_{t})$, we have $r \in R_{*}(s_{k})$, and consequently $r \leq u(s_{k})$. We conclude that $r \leq \alpha$, as desired.

We now show that $\alpha \leq f(x)$.

We know that $-\infty < \alpha$. Take a real number $r < \alpha$. We now argue that $r \leq f(x)$.

There exists an increasing sequence $t_{0} < t_{1} < \cdots$ such that $r < u(s_{t_{k}})$. By discarding finitely many elements of the sequence, we may assume that $-t_{0} < r$. The definition of $u$ now implies that the set $R(s_{t_{k}})$ is not empty for each $k \in \mathbb{N}$, and hence we can take an $r_{k} \in R(s_{t_{k}})$ such that $r < r_{k}$.

Suppose first there exists some $k \in \mathbb{N}$ for which $r_{k} \in R_{*}(s_{t_{k}})$. In that case, $x \in O(s_{t_{k}}) \subseteq \bigcap_{n \in \mathbb{N}}\{g_n>r_{k}\}$. It follows that $r < r_{k} < g_{n}(x)$ for each $n \in \mathbb{N}$ and consequently that $r \leq f(x)$.

Otherwise, for each $k \in \mathbb{N}$ choose an $n_{k} \in \mathbb{N}$ such that $r_{k} \in R_{n_{k}}(s_{t_{k}})$. We have $x \in O(s_{t_{k}}) \subseteq \{g_{n_{k}}>r_{k}\}$ and hence $r < r_{k} < g_{n_{k}}(x)$. It is therefore enough to show that the sequence $n_{0},n_{1},\dots,$ is unbounded: for then the numbers $g_{n_{0}}(x), g_{n_{1}}(x), \dots$ form a sequence converging to $f(x)$, and we are able to conclude that $r \leq f(x)$.

We argue that the sequence $n_{0},n_{1},\dots,$ is unbounded. Suppose to the contrary. By passing to a subsequence, we can then assume that $n_{0} = n_{1} = \cdots$. Now the sequence $r_{0},r_{1},\dots$ is bounded, because $r < r_{k} \leq \inf\{g_{0}(y):y \in O(s_{t_{k}})\} \leq g_{0}(x)$, for each $k \in \mathbb{N}$. Since only finitely many points in the range of $g_{n_{0}}$ fall in the interval $[r,g_{0}(x)]$, only finitely many of the sets $\{\{g_{n_{0}}>r_{k}\}:k \in \mathbb{N}\}$ are distinct. Thus, at least two of these sets are the same, say $\{g_{n_{0}}>r_{0}\} = \{g_{n_{0}}>r_{1}\}$. But $s_{t_{1}}$ is a minimal sequence satisfying $O(s_{t_{1}}) \subseteq \{g_{n_{0}}>r_{1}\}$, while $s_{t_{0}}$ is a proper initial segment of $s_{t_{1}}$ satisfying $O(s_{t_{0}}) \subseteq \{g_{n_{0}} > r_{0}\}$, contradicting $r_1 \in R_{n_1}(s_{t_1})$. Therefore the proof is complete. $\Box$\bigskip

We conclude this section with a list of some properties of the limsup functions that follow easily from the above characterization.
\begin{cor}\label{thm.sum}
The sum, the minimum, and the maximum of two limsup functions is a limsup function. \end{cor}

We say that a collection $\mathcal{C}$ of real-valued functions on $X$ is \textit{closed under pointwise limits from above} if for each sequence $f_{0} \geq f_{1} \geq \cdots$ of functions in $\mathcal{C}$ converging pointwise to a function $f$, the function $f$ is an element of $\mathcal{C}$.

\begin{cor}\label{thm.pointwise}
The set of limsup functions is the smallest collection of functions that (a) contains all lower semicontinuous functions and (b) is closed under pointwise limits from above.
\end{cor}



\begin{cor}\label{thm.uniform}
A uniform limit of limsup functions is a limsup function.
\end{cor}

\begin{cor}\label{thm.Baire1}
A function $f$ is of Baire class 1 if and only if both $f$ and $-f$ are limsup functions.
\end{cor}

\section{A game for limsup functions}\label{secn.gamelimsup}
In this section we turn to the analysis of the game $\Gamma(f)$. Let us begin with the following observation.

\begin{lemma}\label{l:player II wins iff f is limsup}
Player II has a winning strategy in $\Gamma(f)$ precisely when $f$ is a limsup function. 
\end{lemma}

\begin{proof}
From any fixed strategy of Player II in $\Gamma(f)$, one can construct a function $u : T \to \mathbb{R}$ the following way: let Player I play the elements of $(x_0, \dots, x_t) \in T$ and let Player II follow that strategy. Then $u(x_0, \dots, x_t) = v_t$, where $v_0, \dots, v_t$ are the the moves of Player II determined by the strategy. It is easy to check that if a strategy is winning, then the corresponding function $u$ satisfies \eqref{eqn.limsup}, hence $f$ is a limsup function.

Conversely, if $f$ is a limsup function, then let $u$ be the function satisfying \eqref{eqn.limsup}. One can construct a strategy using $u$ by letting Player II respond $u(x_0, \dots, x_t)$ if Player I plays $(x_0, \dots, x_t)$. It is then straightforward to check that the strategy of Player II defined this way is winning.
\end{proof}

Unlike the eraser game (see Kiss  \cite{Kiss[2019]}), as we will show below, the game $\Gamma(f)$ need not be determined.


\smallskip
Recall that a set $B \subseteq X$ is called a \textit{Bernstein set} if neither $B$ nor $X \setminus B$ contains a non-empty perfect set, and also recall that every uncountable Polish space contains a Bernstein set. Moreover, every uncountable analytic set (in a Polish space) contains a non-empty perfect set.

\begin{thm}\label{thm.nondet}
If $X$ is uncountable and $B \subseteq X$ is a Bernstein set, then $\Gamma(1_{B})$ is not determined.
\end{thm}
\begin{proof}
Notice first that $B$ is not a Borel set: for if it were, either $B$ or $X \setminus B$ would contain a non-empty perfect set. And since $B$ is not Borel, the function $1_{B}$ is not a limsup function, and hence Player II has no winning strategy in $\Gamma(1_{B})$.

Suppose that Player I does.

Let $E = \{v \in 2^{\mathbb{N}}:v_{n} = 1\text{ for infinitely many }n\in\mathbb{N}\}$, where we write $2 = \{0,1\}$. Notice that $E$ is not a $\mathbf{\Sigma}_{2}^{0}$ subset of $2^{\mathbb{N}}$. For it were, we would be able to express $2^{\mathbb{N}}$ as a countable union of meagre sets, contradicting the Baire category theorem.

Now, consider the continuous function $g : 2^{\mathbb{N}} \to X$ induced by Player I's winning strategy. Then $g(E) \subseteq X \setminus B$ and $g(2^{\mathbb{N}} \setminus E) \subseteq B$. Since $g(E)$ is an analytic subset of $X$, it is either countable, or it contains a perfect subset. But $X \setminus B$ contains no perfect subset. Thus $g(E)$ is countable, hence a $\mathbf{\Sigma}_{2}^{0}$ subset of $X$. But then $E = g^{-1}(g(E))$ is a $\mathbf{\Sigma}_{2}^{0}$ subset of $2^{\mathbb{N}}$, yielding a contradiction.
\end{proof}


We now turn to a sufficient condition for Player I to have a winning strategy. This sufficient condition will also turn out to be necessary under various assumptions.

Recall that a set is called \textit{a Cantor set} if it is homeomorphic to the classical middle-thirds Cantor set.

\begin{thm}\label{thm.PlayerI}
Let $f : X \to \mathbb{R}$ be arbitrary and suppose that there is a number $r \in \mathbb{R}$ and a Cantor set $C \subseteq X$ such that, in the subspace topology of $C$, the set $C \cap \{f \ge r\}$ is meagre and dense. Then Player I has a winning strategy in $\Gamma(f)$.
\end{thm}
\begin{proof}
Let $Y = C \cap \{f \ge r\}$, and let $\{S_{0},S_{1},\dots\}$ be a cover of $Y$ by closed nowhere dense subsets of $C$. We presently construct a winning strategy for Player I.

Fix some sequence $v_{0},v_{1},\dots$ of Player II's moves.

Let $y(0)$ be any point of the set $Y \setminus S_{0}$. Notice that the set $C \setminus S_{0}$ is not empty because $S_{0}$ is nowhere dense in $C$, and $Y \setminus S_{0}$ is not empty since $Y$ is dense in $C$. Set $m_{0} = 0$.

Player I starts with a move $x_{0} = y(0)_{0}$. Take an $n \in \mathbb{N}$ and suppose that Player I's moves $x_{0}, \dots, x_{n}$ have been defined, along with a point $y(n) \in Y$ and a number $m_{n} \in \mathbb{N}$, such that
\begin{equation}\label{eqn.ind}
(x_{0},\dots,x_{n}) = (y(n)_0,\dots,y(n)_n).
\end{equation}
To define the next move of Player I, $x_{n+1}$, we distinguish two cases:

Case 1: $v_{n} > r - 2^{-m_{n}}$ and $O(x_{0},\dots,x_{n}) \cap S_{m_{n}} = \emptyset$. Let $y(n+1)$ be any point of the set $(O(x_{0},\dots,x_{n}) \cap Y) \setminus S_{m_{n}+1}$. Notice that the set $(O(x_{0},\dots,x_{n}) \cap C) \setminus S_{m_{n}+1}$ is not empty because $S_{m_{n}+1}$ is nowhere dense in $C$, and $(O(x_{0},\dots,x_{n}) \cap Y) \setminus S_{m_{n}+1}$ is not empty since $Y$ is dense in $C$. Let $m_{n+1} = m_{n} + 1$, and define Player I's move as $x_{n+1} = y(n+1)_{n+1}$.

Case 2: otherwise. In this case we let $y(n+1) = y(n)$, $m_{n+1} = m_{n}$, and define Player I's move as $x_{n+1} = y(n+1)_{n+1}$.

Notice that in either case \eqref{eqn.ind} holds for $n+1$. This completes the definition of Player I's strategy.

The intuition behind this definition could be explained as follows: Player I starts by zooming in on the point $y(0)$ chosen to be in $Y$ but not in $S_{0}$. Player I awaits a stage $n$ where Player II would make a move $v_{n} > r - 1$, and where the set $S_{0}$ would be "excluded". As soon as such a stage is reached, Player I switches to an element $y(1)$, chosen to be in $Y$ but not in $S_{1}$. He then zooms in on $y(1)$, awaiting a stage where Player I would make a move $v_{n} > r - 1/2$, and where $S_{1}$ would be "excluded". As soon as such a stage occurs, Player I switches to an element $y(2)$ chosen to be in $Y$ but not in $S_{2}$. And so on.

We argue that the Player I's strategy is winning.

Suppose first that $\limsup v_{n} \leq r - 2^{-m}$ for some $m \in \mathbb{N}$. Then Case 1 occurs at most finitely many times. Let $N$ be the last stage when Case 1 occurs (or $N = 0$ if Case 1 never occurs). Then the point $x$ produced by Player I equals to $y(N)$. We thus have $\limsup v_{n} \leq r - 2^{-m} < r \leq f(y(N)) = f(x)$.

Suppose now that $\limsup v_{n} \geq r$. We argue that Case 1 occurs infinitely many times. Suppose to the contrary and let $N$ be the last stage when Case 1 occurs (or $0$ if Case 1 never occurs). Then $m_{N} = m_{N+1} = \cdots$ and $y(N) = y(N+1) = \cdots = x$. There are infinitely many $n > N$ with $v_{n} > r - 2^{-m_{N}}$, and for each such $n$ the neighborhood $O(x_{0},\dots, x_{n})$ of $x$ has a point in common with $S_{m_{N}}$. This implies that $x \in S_{m_{N}}$. This, however, contradicts the choice of $y(N)$. This establishes that Case 1 occurs infinitely many times.

Let $x$ be the point constructed by Player I. We argue that $x \in C \setminus Y$. In view of \eqref{eqn.ind}, $x$ is a limit of the sequence $y(0),y(1),\dots$. Since each $y(n)$ is an element of the closed set $C$, so is $x$. To see that $x$ is not an element of $Y$, suppose to the contrary. Then $x \in S_{m}$ for some $m \in \mathbb{N}$. Since Case 1 occurs infinitely often, the sequence $m_{0},m_{1},\dots$ runs through all natural numbers, so we can choose $n \in \mathbb{N}$ to be the largest number such that $m_{n} = m$. This choice implies that Case 1 occurs at stage $n$, and hence $O(x_{0},\dots,x_{n})$ is disjoint from $S_{m_{n}}$, leading to a contradiction.

It follows that $x$ is not an element of $\{f \ge r\}$. Thus $\limsup v_{n} \geq r > f(x)$, which completes the proof.
\end{proof}

\begin{rem}
For an arbitrary set $H \subseteq X$ the existence of a Cantor set $C \subseteq X$ such that $C \cap H$ is meager and dense in $C$ is equivalent to the existence of a Cantor set $C \subseteq X$ such that $C \cap H$ is countable and dense in $C$. This either follows from Theorem \ref{thm.PlayerI} and Theorem \ref{thm.PlayerIconverse} applied to $1_H$ and $r=\frac12$, or can also be proved directly by a standard Cantor scheme construction.
\end{rem}

If the range of the function $f$ does not contain a strictly increasing sequence, the condition of Theorem \ref{thm.PlayerI} is both sufficient and necessary for Player I to have a winning strategy. The proof relies on a Kechris--Louveau--Woodin separation theorem (Kechris \cite[Theorem 21.22]{Kechris[1995]}).

\smallskip
For a set $R \subseteq \mathbb{R}$, and function  $f: X \to \mathbb{R}$ define the game $\Gamma_R(f)$ similarly to $\Gamma(f)$, but allowing  Player II to choose $v_i$'s from $R$ instead of the whole real line.

\begin{lemma}\label{l:winning strategy in gamma_R}
  Let  $R \subseteq \mathbb{R}$ and $f: X \to R$. Then Player I has a winning strategy in the game $\Gamma_R(f)$ if and only if Player I has a winning strategy in $\Gamma(f)$.
\end{lemma}
\begin{proof}
It is straightforward to check that if Player I has a winning strategy in $\Gamma(f)$ then the restriction of this strategy is winning for Player I in $\Gamma_R(f)$. 

Conversely, fix a winning strategy $\sigma_R$ of Player I for the game $\Gamma_R(f)$. For each $n \in \mathbb{N}$ define $F_n: \mathbb{R} \to R$ so that
  \begin{equation} \label{Fn}
    \forall y \in \mathbb{R}, \ n \in \mathbb{N}: \  |F_n(y) - y | < \text{d}(y,R) + \tfrac{1}{n} 
  \end{equation}
holds. Now define Player I's strategy $\sigma$ in $\Gamma(f)$ as follows: let $\sigma(\emptyset) = \sigma_R(\emptyset)$, and let
\begin{equation}\label{nyero} 
\sigma(x_0,v_0,x_1,v_1, \dots, x_n,v_n) = \sigma_R(x_0, F_0(v_0), x_1, F_1(v_1), \dots, x_n, F_n(v_n))
\end{equation}
whenever $n \in \mathbb{N}$ and $(x_0, \dots, x_n) \in T$.  

It remains to check that $\sigma$ is a winning strategy for Player I in $\Gamma(f)$.
Fix a run $x_0,v_0,x_1,v_1, \dots$ of the game $\Gamma(f)$ consistent with $\sigma$, i.e. such that for each $n \in \mathbb{N}$, $\sigma(x_0,v_0, \dots, x_n,v_n) = x_{n+1}$. Then $\eqref{nyero}$ implies that for each $n$
\begin{equation} \sigma_R(x_0, F_0(v_0), x_1, F_1(v_1), \dots, x_n, F_n(v_n)) = x_{n+1}, \end{equation}
  and as $\sigma_R$ is a winning strategy for Player I in $\Gamma_{R}(f)$, we obtain that
  \[ f(x_0,x_1,x_2, \dots) \neq \limsup_{n \to \infty} F_n(v_n). \]
  We have to check that $f(x_0,x_1,x_2, \dots) \neq \limsup_{n \to \infty} v_n$. First, if $\limsup_{n \to \infty} v_n \notin R \supseteq \text{ran}(f)$, we are done. Otherwise $\limsup_{n \to \infty} v_n = r \in R$, therefore for each $\varepsilon > 0$, for all but finitely many $k$ we have $v_k < r + \varepsilon$, thus $\eqref{Fn}$ implies $F_k(v_k) < r + 2\varepsilon + \frac{1}{k}$ for these cofinitely many $k$'s,
  therefore $\limsup_{k \to \infty} F_k(v_k) \leq r$. This argument also shows that as for infinitely many $k$, $v_k > r-\varepsilon$ holds, we have $F_k(v_k) > r - 2\varepsilon - \frac{1}{k}$ for infinitely many $k$ too, thus \[ \limsup_{n \to \infty} v_n = r = \limsup_{n \to \infty} F_n(v_n) \neq f(x_0,x_1,x_2, \dots), \] as desired.
\end{proof}


\begin{thm}\label{thm.PlayerIconverse}
Consider a function $f : X \to \mathbb{R}$ such that the range of $f$ contains no infinite strictly increasing sequence. If Player I has a winning strategy in $\Gamma(f)$, then there is a number $r \in \mathbb{R}$ and a Cantor set $C \subseteq X$ such that the set $C \cap \{f \ge r\}$ is countable and dense in $C$. 
\end{thm}

\begin{proof}
Define $R$ to be the closure of the range of $f$. Then it is easy to verify that $R$ contains no infinite strictly increasing sequence. Hence, the usual order $>$ of the reals is a well ordering of $R$. Let $\rho$ be the order type of $(R,>)$, and let $\alpha \mapsto r_{\alpha}$ be the bijective map from $\rho$ to $R$ such that $r_{\alpha} > r_{\beta}$ whenever $\alpha < \beta$. Notice that $\rho$ is a countable ordinal.


Assume that Player I has a winning strategy in $\Gamma(f)$. Then by Lemma \ref{l:winning strategy in gamma_R} Player I also has a winning strategy in $\Gamma_R (f)$. Let $\sigma_R$ be such a strategy. Let  $g : R^{\mathbb{N}} \to X$ be the continuous function induced by $\sigma_R$. Here $R$ is given its discrete topology; since $R$ is countable, $R^{\mathbb{N}}$ is a Polish space. For each $r \in R$ let $L_{r} = \{v \in R^{\mathbb{N}}: \limsup_{t \to \infty}v_{t} = r\}$, and let $A_{r} = g(L_{r})$. The set $A_{r}$ is analytic. Moreover, 
\begin{equation}\label{eqn.disjoint}
\{f = r\} \cap A_{r} = \emptyset.  
\end{equation}

Suppose that the function $f$ fails to satisfy the conclusion of the theorem, that is, there is no number $r\in \mathbb{R}$ and Cantor set $C\subseteq X$ such that $C\cap \{f\ge r\}$ is countable and dense in $C$. We obtain a contradiction by showing that there exists a limsup function $e : X \to R$ such Player I has a winning strategy in the game $\Gamma(e)$. More precisely, we show that $\sigma_R$ is a winning strategy for Player I in $\Gamma_R(e)$, which suffices by Lemma \ref{l:winning strategy in gamma_R}.

Note that a function $e : X \to R$ is a limsup function if and only if $\{e \ge r\}$ is a $\mathbf{\Pi}_{2}^{0}$ set for each $r \in R$, using that $R$ is closed. 

We define recursively a sequence $(G_{\alpha}:\alpha < \rho)$ of $\mathbf{\Pi}_{2}^{0}$ subsets of $X$ such that $\{f \ge r_\alpha\} \subseteq G_{\alpha}$. Let $\beta < \rho$ be an ordinal such that the sets $(G_{\alpha}:\alpha < \beta)$ have been defined. In particular, notice that
\begin{equation}\label{eqn.disjoint1}
\{f > r_{\beta}\} \subseteq \bigcup_{\alpha < \beta}\, \bigcap_{\gamma:\, \alpha \leq \gamma <\beta} G_{\gamma}. 
\end{equation}

Since $f$ fails to satisfy the condition of the theorem, there exists no Cantor set $C \subseteq X$ such that $C \cap \{f \ge r_\beta \}$ is countable and dense in $C$. This implies (using Kechris \cite[Theorem 21.22]{Kechris[1995]}) that $\{f \ge r_\beta \}$ can be separated from any disjoint analytic subset of $X$ by a $\mathbf{\Pi}_{2}^{0}$ set. Consider the set 
\begin{equation}\label{eqn.set}
A_{r_{\beta}} \Big\backslash \bigcup_{\alpha < \beta}\, \bigcap_{\gamma:\, \alpha \leq \gamma <\beta} G_{\gamma}.
\end{equation}
It is analytic, since $\beta$ is a countable ordinal. Moreover, it is 
disjoint from $\{f \ge r_\beta\}$ as can be seen from \eqref{eqn.disjoint} and \eqref{eqn.disjoint1}. Hence there exists a $\mathbf{\Pi}_{2}^{0}$ subset $G_{\beta}$ of $X$ containing $\{f \ge r_\beta\}$ and disjoint from \eqref{eqn.set}. Thus $\{f \ge r_\beta\} \subseteq G_{\beta}$ and 
\begin{equation}\label{eqn.disjoint3}
G_{\beta} \cap A_{r_{\beta}} \subseteq \bigcup_{\alpha < \beta}\, \bigcap_{\gamma:\, \alpha \leq \gamma <\beta} G_{\gamma}.
\end{equation}
This concludes the recursive definition of the sequence $(G_\alpha: \alpha <\rho)$.

Now, for an arbitrary $\beta < \rho$ define the set 
\[E_{\beta} = \bigcap_{\gamma:\, \beta \leq \gamma < \rho} G_{\gamma}.\]
This is a $\mathbf{\Pi}_{2}^{0}$ set, since $\rho$ is a countable ordinal. Moreover, $\{f \ge r_\beta \} \subseteq E_{\beta}$ for each $\beta < \rho$, hence $X = \bigcup_{\beta<\rho}E_{\beta}$. In view of \eqref{eqn.disjoint3}, we have
\begin{equation}\label{eqn.disjoint2}
E_{\beta} \cap A_{r_{\beta}} \subseteq \bigcup_{\alpha < \beta} E_{\alpha}.
\end{equation}
Define $e \colon X \to R$ by letting $e(x) = r_{\beta}$, where $\beta < \rho$ is the smallest ordinal such that $x \in E_{\beta}$. Then $\{e \ge r_\beta\} = E_{\beta}$ for each $\beta<\rho$, so $e$ is a limsup function. Moreover, 
$\{e = r_{\beta}\}$ equals the set $E_{\beta} \setminus \bigcup_{\alpha < \beta}E_{\alpha}$, which is disjoint from $A_{r_{\beta}}$ by \eqref{eqn.disjoint2}. This shows that Player I's strategy $\sigma_R$ remains winning in the game $\Gamma_{R}(e)$, yielding the desired contradiction.
\end{proof}

Next we show that the above necessary and sufficient condition for the existence of a winning strategy for Player I also holds if we assume that $f$ is sufficiently definable, e.g. Borel measurable. 

We say that the function $f$ is \textit{semi-Borel} if for each $r \in \mathbb{R}$, the set $\{f \ge r \}$ is co-analytic.

\begin{thm}
\label{thm.PlayerIconverse2}
Let $f : X \to \mathbb{R}$ be semi-Borel. If Player I has a winning strategy in $\Gamma(f)$, then there is a number $r \in R$ and a Cantor set $C \subseteq X$ such that the set $C \cap \{f \ge r\}$ is countable and dense in $C$. 
\end{thm}

\begin{proof}
If for each $r \in \mathbb{R}$, $\{f \ge r \}$ is a $\mathbf{\Pi}_{2}^{0}$ set, then $f$ is a limsup function, hence Player II has a winning strategy in $\Gamma(f)$, a contradiction. Hence $\{f \ge r \}$ is not a $\mathbf{\Pi}_{2}^{0}$ set for some $r \in \mathbb{R}$. Then the Hurewicz Theorem (see e.g. Kechris \cite[Theorem 21.18]{Kechris[1995]}) implies that there is a Cantor set $C$ such that the set $C \cap \{f \ge r\}$ is countable and dense in $C$.
\end{proof}

\begin{cor}\label{cor.semi-Borel->determined}
If $f$ is semi-Borel, then the game $\Gamma(f)$ is determined.
\end{cor}

\begin{proof}
If for each $r \in \mathbb{R}$, $\{f \ge r \}$ is a $\mathbf{\Pi}_{2}^{0}$ set, then $f$ is a limsup function, hence Player II has a winning strategy in $\Gamma(f)$. Otherwise, $\{f \ge r \}$ is not a $\mathbf{\Pi}_{2}^{0}$ set for some $r \in \mathbb{R}$, and as above, Hurewicz Theorem  implies that there is a Cantor set $C$ such that the set $C \cap \{f \ge r\}$ is countable and dense in $C$, therefore Player I has a winning strategy by Theorem  \ref{thm.PlayerI}.
\end{proof}

Next we will show that in general the condition of Theorem \ref{thm.PlayerI} is not equivalent to the existence of a winning strategy for Player I. More precisely, we will show in Corollary \ref{c:asc} that the restriction on the range of $f$ in Theorem \ref{thm.PlayerIconverse} is optimal; if $R \subseteq \mathbb{R}$ contains an infinite strictly increasing sequence then there exists a function $f : \mathbb{N}^\mathbb{N} \to R$ such that Player I has a winning strategy in $\Gamma(f)$, and $C \cap \{f \ge r\}$ is either uncountable or empty for each $r \in \mathbb{R}$ and each Cantor set $C \subseteq \mathbb{N}^\mathbb{N}$.

\begin{thm}
  \label{t:example with countable range}
  There exists a function $f : \mathbb{N}^\mathbb{N} \to \mathbb{N}$ such that Player I has a winning strategy in $\Gamma(f)$, and $C \cap \{f \ge r\}$ is uncountable for each $r \in \mathbb{R}$ and each Cantor set $C \subseteq \mathbb{N}^\mathbb{N}$. 
\end{thm}

\begin{proof}
First note that every Cantor set can be written as a disjoint union of uncountably many (in fact, continuum many) Cantor sets, since it is well-known that a Cantor set is homeomorphic to $2^I$ for every countably infinite set $I$, in particular to $2^{\mathbb{N} \times \mathbb{N}}$, which is homeomorphic to $2^\mathbb{N} \times 2^\mathbb{N} = \bigcup_{c \in 2^\mathbb{N}} \left(\{c\} \times 2^\mathbb{N} \right)$.

This implies that if $H$ is an arbitrary set, then in order to show that $C \cap H$ is uncountable for each Cantor set $C \subseteq \mathbb{N}^\mathbb{N}$, it suffices to show that
$C \cap H \neq \emptyset$ for each Cantor set $C \subseteq \mathbb{N}^\mathbb{N}$. 

Let $X = \mathbb{N}^{\mathbb{N}}$ and let $\varphi: X \to \mathbb{N} \cup \{+\infty\}$ be given by $\varphi(x) = \limsup_{n \to \infty}x_{n}$. We first argue that there exists a function $f : X \to \mathbb{N}$ such that (a) $f(x) \neq \varphi(x)$ for each $x \in X$, and (b) $C \cap \{f \ge r\} \neq \emptyset$ for each $r \in \mathbb{R}$ and each Cantor set $C \subseteq X$. We will then show that condition (a) implies that Player I has a winning strategy in $\Gamma(f)$, while we already argued that (b) implies that $C \cap \{f \ge r\}$ is uncountable for each $r \in \mathbb{R}$ and each Cantor set $C \subseteq \mathbb{N}^\mathbb{N}$.
  
Let $(r_{\alpha}:\alpha<\mathfrak{c})$, and $(C_{\alpha}:\alpha<\mathfrak{c})$ be enumerations of the real numbers and of the Cantor subsets of $X$, respectively. We define the pairs $(z_{\alpha},f(z_{\alpha})) \in X \times \mathbb{N}$ recursively as follows. Take an ordinal $\alpha < \mathfrak{c}$ and suppose that $(z_{\beta},f(z_{\beta}))$ has been defined for every $\beta < \alpha$. Let $z_{\alpha}$ be any point of $C_{\alpha} \setminus \{z_{\beta}:\beta<\alpha\}$. Define $f(z_{\alpha})$ to be the smallest natural number such that $f(z_{\alpha}) \geq r_{\alpha}$ and $f(z_{\alpha}) \neq \varphi(z_{\alpha})$. To complete the definition of $f$, for each point $x \in X \setminus \{z_{\beta}:\beta<\mathfrak{c}\}$ let $f(x)$ be the smallest natural number such that $f(x) \neq \varphi(x)$.
  
Now we show that Player I has a winning strategy in $\Gamma(f)$. Using Lemma \ref{l:winning strategy in gamma_R}, it is enough to show that Player I has a winning strategy in $\Gamma_{\mathbb{N}}(f)$. Let Player I start by playing $x_{0} = 0$. To a move $v_{n} \in \mathbb{N}$ of Player II in round $n$, Player I responds with $x_{n+1} = v_n$. Then, for a run $x_{0},v_{0},x_{1},v_{1},\dots$ of the game, it holds that $\varphi(x) = \limsup_{n \to \infty}x_{n} = \limsup_{n \to \infty}v_n$. Since $f(x) \neq \varphi(x)$, $f(x) \neq \limsup_{n \to \infty} v_n $, therefore the run is won by Player I. 
\end{proof}

\begin{cor}
\label{c:asc}
 If $R \subseteq \mathbb{R}$ contains an infinite strictly increasing sequence, then there exists a function $f : \mathbb{N}^\mathbb{N} \to R$ such that Player I has a winning strategy in $\Gamma(f)$, and $C \cap \{f \ge r\}$ is either uncountable or empty for each $r \in \mathbb{R}$ and each Cantor set $C \subseteq \mathbb{N}^\mathbb{N}$.
\end{cor}

\begin{proof}
Let $i : \mathbb{N} \to R$ be a strictly increasing map. Let $f_0 : \mathbb{N}^\mathbb{N} \to \mathbb{N}$ be a function as in Theorem \ref{t:example with countable range}, that is, such that Player I has a winning strategy in $\Gamma(f_0)$, and $C \cap \{f_0 \ge r\}$ is uncountable for each $r \in \mathbb{R}$ and each Cantor set $C \subseteq \mathbb{N}^\mathbb{N}$. We claim that the function defined as $f = i \circ f_0$ works. Clearly, $f : \mathbb{N}^\mathbb{N} \to R$, and it is also clear that $C \cap \{f \ge r\}$ is either uncountable or empty for each $r \in \mathbb{R}$ and each Cantor set $C \subseteq \mathbb{N}^\mathbb{N}$, hence we only have to show that Player I has a winning strategy in $\Gamma(f)$. By Lemma \ref{l:winning strategy in gamma_R} it suffices to check that Player I has a winning strategy in $\Gamma_{i(\mathbb{N})}(f)$.
Let $\sigma_0$ be a winning strategy for Player I in $\Gamma(f_0)$, and define
\begin{equation}
\label{nyero2}
\sigma_{i (\mathbb{N} )} (x_0,v_0, \dots, x_n,v_n) = \sigma_0(x_0, i^{-1}(v_0), \dots, x_n, i^{-1}(v_n)) \ \ (\forall n \in \mathbb{N}).
\end{equation}
Since $i$ is order-preserving, it is easy to check that $\sigma_{i (\mathbb{N} )}$ is a winning strategy for Player I in 
$\Gamma_{i(\mathbb{N})}(f)$.
\end{proof}

Next we state another result of similar sort.
We will strengthen the above counterexamples by showing that such an $f$ can have a co-analytic graph, but on the other hand we have to sacrifice that the range is countable. Note that the complexity of the graph of $f$ is optimal, since if the graph of a function is analytic, then it is well-known that the function is actually Borel measurable, hence by Theorem \ref{thm.PlayerIconverse2} it cannot be a counterexample, and similarly, the range cannot be countable, since it is easy to show that a function with co-analytic graph and countable range is semi-Borel.

\smallskip
Recall that the statement $"V=L"$ is the \textit{Axiom of Constructibility} due to K. Gödel. It is known that it is consistent with $ZFC$, and that it implies the Continuum Hypothesis.

\begin{thm}
  \label{t:consistent example with coanalytic graph}
  Assume $V=L$. Then there exists a function $f : \mathbb{N}^\mathbb{N} \to \mathbb{R}$ with co-analytic graph such that Player I has a winning strategy in $\Gamma(f)$, and $C \cap \{f \ge r\}$ is uncountable for each $r \in \mathbb{R}$ and each Cantor set $C \subseteq \mathbb{N}^\mathbb{N}$. 
\end{thm}


\begin{proof}
  Let $X = \mathbb{N}^\mathbb{N}$. Let $q(0), q(1), \dots$ be an enumeration of the rational numbers, and let $\varphi : X \to \mathbb{R}\cup \{+\infty\}$ be given by $\varphi(x) = \limsup_{n \to \infty} q(x_n)$.
  In order to construct $f : X \to \mathbb{R}$ with co-analytic graph, we use a result of  Vidny\'anszky \cite[Theorem 1.3]{Vidnyanszky[2014]}. Let $B_1 = \{ (C, t)  \in \mathcal{K}(X) \times \mathbb{R} : C$ is a Cantor set$\}$, where $\mathcal{K}(X)$ is the family of non-empty compact sets in $X$ equipped with the Hausdorff metric. Let $B_2 = \mathbb{R}$ and $B = B_1 \sqcup B_2$ be the disjoint union of $B_1$ and $B_2$ making $B$ a subset of the Polish space $(\mathcal{K}(X)\times \mathbb{R}) \sqcup \mathbb{R}$. Let $i : X \to \mathbb{R}$ be a Borel bijection,  $M = \mathbb{R}^2$, and let 
  $$F_1 = \left\{\big(A, (C, r), (y, t)\big) \in M^{\le \omega} \times B_1 \times M : y \in i(C) \setminus \textrm{pr}_1(\textrm{ran}(A)), t \ge r, t \neq \varphi(i^{-1}(y))\right\},$$
  where $\textrm{pr}_1(\textrm{ran}(A))$ is the projection of the range of the sequence $A$ onto the first coordinate. Let 
  \begin{align*}
    F_2 = \Big\{\big(A, y', (y, t)\big) \in M^{\le \omega} \times B_2 \times M : \; & t \neq \varphi\big(i^{-1}(y)\big), y' \not\in \textrm{pr}_1\big(\textrm{ran}(A)\big) \Rightarrow y' = y,
    \\ & y' \in \textrm{pr}_1\big(\textrm{ran}(A)\big) \Rightarrow y \not\in \textrm{pr}_1\big(\textrm{ran}(A) \big)\Big\},
  \end{align*}
  and let $F = F_1 \sqcup F_2 \subseteq M^{\le \omega} \times B \times M$.
  
  We now check that the conditions of Vidny\'anszky's theorem are satisfied. First, a non-empty compact set $C \subseteq \mathbb{N}^\mathbb{N}$ is a Cantor set if and only if it is perfect. Using Kechris  \cite[4.31]{Kechris[1995]} one can easily see that $B_1$ is a Borel subset of $\mathcal{K}(X) \times \mathbb{R}$. Therefore $B$ is a Borel subset of $(\mathcal{K}(X) \times \mathbb{R}) \sqcup \mathbb{R}$. The set $F_1$ is clearly co-analytic, and since $A \in M^{\le \omega}$ is a countable sequence, conditions of the form $y' \in \textrm{pr}_1(\textrm{ran}(A))$ are Borel. Therefore $F_2$ is even Borel, making $F = F_1 \sqcup F_2$ co-analytic. For each $(A, b) \in M^{\le \omega} \times B$, no matter whether $b \in \mathcal{K}(X) \times \mathbb{R}$ or $b \in \mathbb{R}$, the section 
  $$F_{(A, b)} = \{(y, t) \in M : \big(A, b, (y, t)\big) \in F\}$$ contains $\{x_1\} \times \{t : t \ge x_2\}$ for some $(x_1, x_2) \in \mathbb{R}^2$, hence it is cofinal in the Turing degrees (for this notion, see Definition~1.1 of Vidny\'anszky  \cite{Vidnyanszky[2014]}). Therefore the conditions of the theorem are satisfied. 
  
  The conclusion of the theorem assures that there is a co-analytic set $G \subseteq M = \mathbb{R}^2$ and enumerations $B = \{b_\alpha : \alpha < \omega_1\}$, $G = \{g_\alpha : \alpha < \omega_1\}$ and for every $\alpha < \omega_1$ a sequence $A_\alpha \in M^{\le \omega}$ that is an enumeration of $\{g_\beta : \beta < \alpha\}$ such that $g_\alpha \in F_{(A_\alpha, b_\alpha)}$ for every $\alpha < \omega_1$. We note here that the assumption $V = L$ implies the continuum hypothesis. 
  
  First we check that $G$ is the graph of a function with domain $\mathbb{R}$.
  Notice that for $\beta < \alpha$, if $g_\alpha = (y_1, t_1)$ and $g_\beta = (y_2, t_2)$ then $y_1 \neq y_2$. Indeed, $g_\alpha \in F_{(A_\alpha, b_\alpha)}$ implies that $y_1 \not \in \textrm{pr}_1(\textrm{ran}(A_\alpha))$, and since $y_2 \in \textrm{pr}_1(\textrm{ran}(A_\alpha))$, $y_1 \neq y_2$ easily follows.
  To see that for each $y \in \mathbb{R}$, $(y, t) \in G$ for some $t \in \mathbb{R}$, let $\alpha < \omega_1$ be chosen with $b_\alpha = y \in B_2$. Then either $y \in \textrm{pr}_1(\textrm{ran}(A_\alpha))$ and we are done, or $g_\alpha$ is chosen to be $(y, t)$ for some $t \in \mathbb{R}$. Therefore $G$ is indeed a graph of a function with domain $\mathbb{R}$.
  
  Now we define the function $f : X \to \mathbb{R}$ the following way: for each $(y, t) \in G$, let $f(i^{-1}(y)) = t$. Clearly, the graph of $f$ is $(i, \textrm{id})^{-1}(G)$, hence it is co-analytic.
  
  We now show that the defined function $f$ has properties (a) $f(x) \neq \varphi(x)$ for each $x \in X$, and (b) $C \cap \{f \ge r\} \neq \emptyset$ for each $r \in \mathbb{R}$ and each Cantor set $C \subseteq X$. Then we will show that (a) implies that Player I has a winning strategy in $\Gamma(f)$. The proof that (b) implies that $C \cap \{f \ge r\}$ is uncountable for each $r \in \mathbb{R}$ and each Cantor set $C$ is exactly the same as in the proof of Theorem \ref{t:example with countable range}. 
  
  To show (a), let $(x, t) \in X \times \mathbb{R}$ be a pair with $(i(x), t) = g_\alpha \in G$. Then $g_\alpha \in F_{(A_\alpha, b_\alpha)}$ implies $t \neq \varphi(x)$, hence $f(x) = t \neq \varphi(x)$. To show (b), let $C \subseteq X$ be a Cantor set and let $r \in \mathbb{R}$. Let $\alpha < \omega_1$ be the ordinal with $b_\alpha = (C, r)$. Then for $g_\alpha = (y, t)$, using again that $g_\alpha \in F_{(A_\alpha, b_\alpha)}$, $y \in i(C)$ and $t \ge r$, hence $i^{-1}(y) \in C$ and $f(i^{-1}(y)) = t \ge r$.
  
  It remains to show that Player I has a winning strategy in $\Gamma(f)$. Let Player I start by playing $x_{0} = 0$. To a move $v_{n}$ of Player II, Player I responds with an $x_{n+1} \in \mathbb{N}$ chosen to be the smallest natural number satisfying $|v_{n} - q(x_{n+1})| \leq 2^{-n}$. Then, for a run $x_{0},v_{0},x_{1},v_{1},\dots$ of the game, it holds that $\varphi(x) = \limsup_{n \to \infty}v_{n}$. Since $f(x) \neq \varphi(x)$, the run is won by Player I. 
\end{proof}

We note that the assumption $V=L$ cannot be simply dropped from the above theorem. Indeed, it can be derived using the standard proof that Projective Determinacy implies that the Hurewicz theorem holds for all projective sets, moreover, if the graph of $f$ is projective then so is $\{f \ge r\}$ for every $r\in\mathbb{R}$. Thus one could derive an analogue to Theorem \ref{thm.PlayerIconverse2} under Projective Determinacy, assuming only that $f$ has a projective graph.

\smallskip
Despite all the partial results above, we still do not know the answer to the following interesting question.

\begin{qst}
For which $f : X \to \mathbb{R}$ does Player I have a winning strategy in $\Gamma(f)$?
\end{qst}

\section{A game for Baire class 1 functions}\label{secn.gamebaire}

Recall the definition of the game $\Gamma'(f)$ from the Introduction.
Corollary \ref{thm.Baire1} immediately yields the following result:
\begin{cor}
\label{thm.BairePlayerII}
Player II has a winning strategy in $\Gamma'(f)$ if and only if Player II has winning strategies in both games $\Gamma(f)$ and $\Gamma(-f)$, if and only if $f$ is of Baire class 1.
\end{cor}

New we turn to the existence of a winning strategy for Player I. 

Let $C \subseteq X$ be a closed set, and consider the restriction of $f$ to $C$. The oscillation of $f|_{C}$ at a point $x \in C$ is defined as
$${\rm osc}_f(C,x) = \inf_{\substack{s \in T:\\x \in O(s)}}\,\sup_{y,z \in O(s) \cap C}|f(y) - f(z)|.$$

\begin{lemma}
Suppose that there is a closed set $C \subseteq X$ such that the oscillation of $f|_{C}$ is bounded away from zero: $\inf_{x \in C}{\rm osc}_f(C,x) > 0$. Then Player I has a winning strategy in $\Gamma'(f)$.
\end{lemma}
\begin{proof}
Assume that ${\rm osc}_f(C,x) \geq 5\epsilon > 0$ for each $x \in C$. We will first describe a strategy of Player I and then we will show that it is a winning strategy. To define the moves of Player I in a particular run, we will use recursion to define natural numbers $n_0<n_1<n_2<\ldots$ and sequences $s_0, s_1, s_2, \ldots \in T$ (these may depend on the moves of Player II).

Let $n_{0} = 0$, and let $s_{0}$ be the empty sequence. Suppose that, for some even number $k \in \mathbb{N}$, Player I's moves prior to the stage $n_{k}$ have been defined.

Let $s_{k} \in T$ denote the sequence of Player I's moves prior to the stage $n_{k}$. Define $\alpha_{k} = \sup\{f(x): x \in O(s_{k}) \cap C\}$, and choose a point $x(k) \in O(s_{k}) \cap C$ so that $\alpha_{k} - \epsilon < f(x(k))$. Starting with the stage $n_{k}$, Player I produces his moves using the point $x(k)$, that is, he plays $x_{n} = x(k)_n$ at a stage $n \geq n_{k}$. He continues doing so until the first stage, say $n_{k+1} > n_{k}$, that Player II makes a move $(v_{n_{k+1}}, w_{n_{k+1}})$ such that $|v_{n_{k+1}} - f(x(k))| < \epsilon$. If no such stage occurs, then Player I goes on using the point $x(k)$ to make his moves until the end of the game.

Let $s_{k+1} \in T$ denote the sequence of moves produced by Player I prior to the stage $n_{k+1}$. Define $\beta_{k+1} = \inf\{f(x): x \in O(s_{k+1}) \cap C\}$, and choose a point $x(k+1) \in O(s_{k+1}) \cap C$ so that $f(x(k+1)) < \beta_{k+1} + \epsilon$. Starting with the stage $n_{k+1}$, Player I produces the moves using $x(k+1)$, that is he plays $x_{n} = x(k+1)_n$ at a stage $n \geq n_{k+1}$. He continues doing so until the first stage, say $n_{k+2} > n_{k+1}$, that Player II makes a move $(v_{n_{k+2}},w_{n_{k+2}})$ such that $|w_{n_{k+2}} - f(x(k+1))| < \epsilon$. If no such stage occurs, then Player I goes on using the point $x(k+1)$ until the end of the game.

We show that the thus defined strategy is winning.

Suppose first that only finitely many stages $n_{0},n_{1},\dots$ occur, the last one being $n_{k}$. For concreteness, suppose that $k$ is even. In this case Player I uses the point $x(k)$ to generate his moves until the end of the game. Moreover, there is no $n > n_{k}$ such that $|v_{n} - f(x(k))| < \epsilon$. This implies that $\limsup v_{n} \neq f(x(k))$, and hence the run is won by Player I. Likewise, if the last one of the sequence $n_{0},n_{1},\dots$ is the stage $n_{k+1}$ where $k$ is even, then Player I generates the point $x(k+1)$, and there is no $n > n_{k+1}$ such that $|w_{n} - f(x(k+1))| < \epsilon$. Therefore $\liminf w_{n} \neq f(x(k+1))$, and hence the run is won by Player I.

Suppose that infinitely many stages $n_{0},n_{1},\dots$ occur. From the above definitions we get for each even $k \in \mathbb{N}$
\begin{align*}
v_{n_{k+1}} &> f(x(k)) - \epsilon\\
&> \alpha_{k} - 2\epsilon\\
&=(\alpha_{k} - \beta_{k+1}) + \beta_{k+1} - 2\epsilon\\
&\geq (\alpha_{k} - \beta_{k+1}) + f(x(k+1)) - 3\epsilon\\
&\geq (\alpha_{k} - \beta_{k+1}) + w_{n_{k+2}} - 4\epsilon.
\end{align*}
Let $\alpha_{k+1} = \sup\{f(x):x \in C \cap O(s_{k+1})\}$. Since the sequence $s_{k+1}$ extends $s_{k}$, we have $\alpha_{k} \geq \alpha_{k+1}$. By the assumption, the oscillation of $f|_{C}$ at the point $x(k+1) \in C$ is at least $5\epsilon$, hence $\alpha_{k+1} - \beta_{k+1} \geq 5\epsilon$. Combining these facts we obtain that for each even $k \in \mathbb{N}$ it holds that $v_{n_{k+1}} \geq w_{n_{k+2}} + \epsilon$. This, however, means that $\limsup v_{n} > \liminf w_{n}$, implying a win for Player I.
\end{proof}

\begin{rem}\rm
The above construction of the winning strategy for Player I is similar to that in Kiss \cite{Kiss[2019]}. In both cases Player I zooms in on a particular element of $C$ until Player II triggers a switch to another element. The main difference is that here Player I undergoes two alternating types of switches: even switches are different from the odd. An odd switch, say $(k+1)$st (where $k$ is even) is triggered when Player II makes a move such that $v_{n}$ is close to $f(x(k))$. Player I reacts by switching to a point $x(k+1)$ with a low value of $f$. Even switches, say $(k+2)$nd, are triggered when Player II makes a move such that $w_{n}$ is close to $f(x(k+1))$. Player I reacts by switching to a point $x(k+2)$ of $C$ with a high value of $f$.
\end{rem}


\begin{thm}
The game $\Gamma'(f)$ is determined.
\end{thm}
\begin{proof}
If $f$ is a function in Baire class 1, then Player II has a winning strategy by Lemma \ref{thm.BairePlayerII}. Suppose that $f$ is not a function in Baire class 1. Then (see Kuratowski \cite[Theorem 2 and Remark 1 on p.395]{Kuratowski[1966]}) there exists a non-empty closed set $K \subseteq X$ such that the set of discontinuity points of $f|_{K}$ contains an open subset of $K$. Using the Baire category theorem and the arguments as in Kiss \cite[p.9]{Kiss[2019]} one can show that there is a non-empty closed set $C \subseteq K$ such that the oscillation of $f|_{C}$ is bounded away from zero. The preceding lemma then implies that Player I has a winning strategy.
\end{proof}

\begin{rem}
It is not completely clear which results of the paper use the countability of $A$ in an essential way. It seems to us that almost all results go through without the assumption that $A$ is countable, and the only really problematic issues are the applications of the Hurewicz theorem and the Kechris--Louveau--Woodin theorem in the proofs of Theorems \ref{thm.PlayerIconverse}, \ref{thm.PlayerIconverse2} and \ref{cor.semi-Borel->determined}. 
\end{rem}


\begin{thebibliography}{}
\bibitem{Bru[2017]}{V\'{e}ronique Bruyère [2017]: Computer Aided Synthesis: A Game-Theoretic Approach. In: Charlier É., Leroy J., Rigo M. (eds) \textit{Developments in Language Theory}. DLT 2017. Lecture Notes in Computer Science, vol 10396. Springer.}

\bibitem{Carroy[2014]}{Raphaël Carroy [2014]: Playing in the first Baire class. \textit{Mathematical Logic Quarterly} 60, 118--132.}

\bibitem{Dubins[2014]}{Lester E. Dubins, Leonard J. Savage, edited and updated by William Sudderth, and David Gilat [2014]: \textit{How to gamble if you must: Inequalities for stochastic processes.} Dover publications, New York.}

\bibitem{Duparc[2001]}{Jacques Duparc [2001]: Wadge hierarchy and Veblen hierarchy. I. Borel sets of finite rank. \textit{Journal of Symbolic Logic} 66, 56--86.}

\bibitem{Hausdorff[2005]}{Felix Hausdorff [2005]: \textit{Set theory}. American Mathematical Soc., Vol. 119.}

\bibitem{Kiss[2017]}{Viktor Kiss [2017]: Classification of bounded Baire class $\xi$ functions. \textit{Fundamenta Mathematicae} 236, 141--160.}

\bibitem{Kiss[2019]}{Viktor Kiss [2019]:  A game characterizing Baire class 1 functions. \textit{Journal of Symbolic Logic}, DOI: https://doi.org/10.1017/jsl.2019.38.}

\bibitem{Kechris[1995]}{Alexander S. Kechris [1995]: Classical Descriptive Set Theory. Springer--Verlag.}

\bibitem{Kuratowski[1966]}{Kazimierz Kuratowski [1966]: \textit{Topology. Volume I}. Second edition. Academic Press.}

\bibitem{Levy[2017]}{Yehuda J. Levy and Eilon Solan [2017]: Stochastic Games. In: Meyers R. (eds) \textit{Encyclopedia of Complexity and Systems Science}. Springer, Berlin, Heidelberg.}

\bibitem{Nobrega[2019]}{Hugo Nobrega [2019]: Games for functions of a fixed Baire class. Preprint.}

\bibitem{Semmes[2008]}{Brian Thomas Semmes [2008]: Games, trees, and Borel functions. Ph.D. Thesis, ILLC, University of Amsterdam, Holland.}

\bibitem{Vidnyanszky[2014]}{Zolt\'an Vidny\'anszky [2014]: Transfinite inductions producing coanalytic sets. \textit{Fundamenta Mathematicae} 224(2), 155--174.}
\end{thebibliography}
\end{document}